\documentclass[11pt]{article}
\usepackage{verbatim}
\usepackage{amsfonts,amssymb,amsmath,bm,a4,color,amsthm}

\newtheorem{theorem}{Theorem}

\newtheorem{corollary}{Corollary}
\newtheorem{lemma}{Lemma}
\theoremstyle{remark}
\newtheorem{remark}{Remark}

\def\R{\mathbb{R}}

\def\Q{\mathbb{Q}}
\def\Z{\mathbb{Z}}
\def\N{\mathbb{N}}

\begin{document}

\title{\bf Sums of reciprocals and the three distance theorem}

\author{\sc Victor Beresnevich\footnote{Supported by EPSRC grant EP/J018260/1} \and \sc Nicol Leong}

\date{}

\maketitle

\begin{abstract}
In this paper we investigate the sums of reciprocals to an arithmetic progression taken modulo one, that is sums of $\{n\alpha-\gamma\}^{-1}$, where $\alpha$ and $\gamma$ are real parameters and $\{\,\cdot\,\}$ is the fractional part of a real number. Bounds for these sums have been studied for a long while in connection with various applications. In this paper we develop an alternative technique for obtaining upper bounds for the sums and obtain new efficient and fully explicit results. The technique uses the so-called three distance theorem.
\end{abstract}

\section{Introduction}

Throughout, given an $x \in \mathbb{R}$, we write $[x]$ for its integer part, $\{ x\}=x-[x]$ for its fractional part and $\|x\|$ for the distance of $x$ from the nearest integer. Throughout $\alpha,\gamma\in\R$ and $\alpha$ is irrational. The key object of interest of this paper is the following sum
\begin{equation}\label{T}
T_N(\alpha,\gamma):=\sum_{\substack{0\le n\le N\\ n\neq n'}}\frac{1}{\{n\alpha-\gamma\}}
\end{equation}
where $N\in\N$ and $0\le n'\le N$ is an integer such that
\begin{align}\label{n'}
\{ n'\alpha -\gamma \} = \min_{0\le n \le N} \{ n\alpha -\gamma \}\,.
\end{align}
Note that $n'$ depends on $\alpha$, $\gamma$ and $N$ and, since $\alpha$ is irrational, $n'$ is unique. One obvious reason for imposing the condition $n\neq n'$ in \eqref{T} is to avoid a vanishing denominator. However, even if $\{n'\alpha-\gamma\}$ does not vanish the term $\{n'\alpha-\gamma\}^{-1}$ cannot be bounded by a function of $N$ and $\alpha$. This is because, for any given $N$ and $\alpha$, the parameter $\gamma$ can be arbitrarily close to one of the points $\{n\alpha\}$ with $0\le n\le N$. Also in this paper, we shall consider the sums with $\{n\alpha-\gamma\}$ replaced by $\|n\alpha-\gamma\|$ and the more general expression $n^{a}\{n\alpha-\gamma\}^{b}$ -- see the last section.

There are various applications of bounds on $T_N(\alpha,\gamma)$, e.g. for counting lattice points in polygons, in the theory of uniform distribution, metric theory of Diophantine approximation and so on. The above sums have therefore been studied in depth for a long while and also recently, see, for example, \cite{Beck1994}, \cite{Beck2014}, \cite{Behnke}, \cite{BHV-MEM}, \cite{Chowla}, \cite{HardyLittlewood1922}, \cite{HardyLittlewood1922b}, \cite{HardyLittlewood1930}, \cite[pp.\,108-110]{Koksma}, \cite{KuipersNiederreiter1974}, \cite{Kruse}, \cite{Vaaler}, \cite{Schmidt1964}, \cite{SinaiUlcigrai2009}, \cite{Walfisz1}, \cite{Walfisz2}.

The purpose of this paper is to introduce yet another technique for studying the above sums and obtain new results as well as recover known bounds in a more precise form. A detailed overview of previously obtained bounds for the sums of reciprocals can be found in the recent paper \cite{BHV-MEM}.
The technique we propose makes use of the so-called Three Distance Theorem, see \S\ref{3d}. This theorem relies on the theory of continued fractions. The facts from the theory of continued fractions we use in this paper can be found in any textbook on the topic, for example, \cite{Khintchine}, \cite{lang} or \cite{sg}.

From now on, unless otherwise stated, $[a_0;a_1,a_2,\dots]$ will denote the simple continued fraction expansion of a fixed real irrational number $\alpha$, where the $a_k$ are partial quotients. Thus, $a_0\in\Z$, $a_1,a_2,\dots\in\N$ are the unique integers such that
$$
\alpha=a_0+\dfrac{1}{a_1+\dfrac{1}{a_2+\raisebox{-2ex}{$\ddots$}}}\,.
$$
Further, $p_k/q_k=[a_0;a_1,a_2,\dots, a_k]$ will denote the $k$th convergent to $\alpha$. Recall that the numbers $p_k$ and $q_k$ are coprime integers that satisfy the following recurrence relations:
\begin{equation}\label{q_k}
q_{k+1} = a_{k+1}q_k + q_{k-1},\qquad p_{k+1} = a_{k+1}p_k + p_{k-1}\qquad(k\ge0)
\end{equation}
with  $p_{0}=a_0$, $p_{-1}=1$, $q_0=1$, $q_{-1}=0$, and
\begin{equation}\label{vb100}
  q_{k+1}p_k-q_kp_{k+1}=(-1)^{k+1}\quad (k\ge-1).
\end{equation}
Our key result concerting $T_N(\alpha,\gamma)$ is the following.

\begin{theorem}\label{main3}
Let $N \in \mathbb{N}$, $\alpha \in \mathbb{R} \backslash \mathbb{Q}$, $p_k/q_k$ denote the convergents to $\alpha$ and $K=K(N,\alpha)$ be the largest integer satisfying $q_K \le N$. Further, let $\gamma \in \mathbb{R}$ and $n'\in[0,N]$ be an integer given by \eqref{n'}.
Then
\begin{equation}\label{e1}
T_N(\alpha,\gamma) ~<~ 4N\left( \log q_K +1\right)~+~2q_{K+1}\left( \log \left( \frac{N}{q_{K}}\right) +1 \right)
\end{equation}
and furthermore
\begin{equation}\label{e2}
\sum_{\substack{0\le n\le N\\ n\not\equiv n'\pmod{q_{K}}}} \frac{1}{\{ n \alpha -\gamma \}} < 4N\left( \log q_K +1\right)\,.
\end{equation}
\end{theorem}

\bigskip

\begin{remark}\label{rem2}
The sequence $(q_k)_{k\ge1}$ of the denominators of the convergents to $\alpha$ is the unique largest increasing sequence of positive integers such that $\|q_1\alpha\|>\|q_2\alpha\|>\|q_3\alpha\|>\dots$.
Since for any positive integer $q$ we have that $\|q(-\alpha)\|=\|q\alpha\|$, we have that $q_k(-\alpha)=q_k(\alpha)$ for all $\alpha$. Therefore  Theorem~\ref{main3} is applicable to $-\alpha$, resulting in the same bounds.
Further for any real number $x$ note that $\|x\|=\min(\{x\},\{-x\})$ and furthermore that $\tfrac12\le\max(\{x\},\{-x\})\le1$. Hence
\begin{equation}\label{TR0}
1<\frac{1}{\{x\}} + \frac{1}{\{-x\}} - \frac{1}{\lVert x\rVert} \le 2\,.
\end{equation}
We shall use the above observations to prove the following
\end{remark}

\bigskip

\begin{corollary}\label{main4}
Let $N$, $\alpha$, $p_k/q_k$ and $K$ be the same as in Theorem~\ref{main3}. Further, let $\gamma \in \mathbb{R}$ and $n^*\in[0,N]$ be an integer such that
\begin{align*}
\| n^*\alpha -\gamma \| = \min_{0\le n \le N}\big\| n\alpha -\gamma\|.
\end{align*}
Then
\begin{align}
\sum_{\substack{0\le n\le N\\ n\neq n^*}}
\frac{1}{\|n\alpha-\gamma\|}~&\le~ 8N( \log q_K +1)
+4q_{K+1}\left( \log \left( \frac{N}{q_{K}}\right) +2 \right).\label{vb1}
\end{align}
\end{corollary}

\begin{proof}
Let $n'$ be the same as in Theorem~\ref{main3} and similarly let $n''$ be
the unique integer in $[0,N]$ such that
$\{ n''(-\alpha) +\gamma \} = \min_{0\le n \le N}\{n(-\alpha) +\gamma\}$.
Then, by Theorem~\ref{main3} and inequalities \eqref{TR0}, we have that
\begin{align}
\sum_{\substack{0\le n\le N \\[0.3ex] n\neq n',n''}} \frac{1}{\|n\alpha-\gamma\|}
&\le \nonumber
\sum_{\substack{0\le n\le N \\[0.3ex] n\neq n'}} \frac{1}{\{n\alpha-\gamma\}}+
\sum_{\substack{0\le n\le N \\[0.3ex] n\neq n''}} \frac{1}{\{-(n\alpha-\gamma)\}}\\
&\le 8N\left( \log q_K +1\right)~+~4q_{K+1}\left( \log \left( \frac{N}{q_{K}}\right) +1 \right)\,.\label{S_2+}
\end{align}
Note that $n^*$ is either $n'$ or $n''$. If $n'\neq n''$ then we also have that $\|n'\alpha-n''\alpha\|\ge\|q_K\alpha\|$ since $q_K$ is the largest best approximation $\le N$ and $1\le |n'-n''|\le N$. Therefore, the $\max\{\|n'\alpha-\gamma\|,\|n''\alpha-\gamma\|\}\ge\tfrac12\|q_K\alpha\|\ge1/(4q_{K+1})$. Hence
$$
\min\left\{\frac{1}{\|n'\alpha-\gamma\|},\frac{1}{\|n''\alpha-\gamma\|}\right\}\le 4q_{K+1}
$$
and combining this with \eqref{S_2+} completes the proof.
\end{proof}

\section{The Three Distance Theorem}\label{3d}

The main question we discuss in this section is the following: \emph{given an $N\in\N$ and $\alpha\in\R$, what can we say about the distribution of the points
\begin{equation}\label{eq1}
\{\alpha\},\{2\alpha\},\dots,\{N\alpha\}
\end{equation}
in the unit interval\/ $[0,1)$?} Equivalently this question can be posed using circle rotations on identifying $[0,1)$ with the unit circle in the usual way via the map $\alpha \mapsto e^{2\pi i \alpha}$. The following statement conjectured by Hugo Steinhaus is widely known as the three distance theorem or three gaps theorem.

\begin{theorem}[The Three Distance Theorem]
For any $\alpha\in\R\setminus\Q$ and any integer $N\ge 1$ the points \eqref{eq1} partition $[0,1]$ into $N+1$ intervals which lengths take at most 3 different values $\delta_A$, $\delta_B$ and $\delta_C$ with $\delta_C=\delta_A+\delta_B$.
\end{theorem}

There are various generalisations of the above fact and several independent proofs, in particular in \cite{halton, mic, tony, slater, sos1, sos2, suranyi, swier}. Remarkably, the length of the gaps as well as the number of gaps of every length can be exactly specified using the continued fraction expansion of $\alpha$.
Within this paper we will use an even more refined statement which in addition specifies the order in which the intervals of each length appear. The theorem will require the theory of continued fractions.
Some basic notation and facts of this theory have already been recalled in the introduction above. In addition, for $k\ge0$, we define the following quantities, which measure how the $k$th convergent approximates $\alpha$:
\begin{equation}\label{D_k}
D_k=q_k\alpha-p_k.
\end{equation}
By \eqref{q_k} and the definition of $D_k$, we clearly have that
\begin{align}
a_{k+1}D_k = D_{k+1} - D_{k-1} \qquad (k\ge 1). \label{rel2}
\end{align}
The following well known statement, that can be found in \cite{lang}, reveals the approximation properties of the convergents to $\alpha$:

\begin{lemma}\label{cf}
The sequence of even convergents $p_{2k}/q_{2k}$ is strictly increasing, the sequence of odd convergents $p_{2k+1}/q_{2k+1}$ is strictly decreasing. Both converge to $\alpha$ and furthermore we have that
\begin{align}\label{cf1}
\frac{1}{2q_{k+1}} < \frac{1}{q_{k+1}+q_k} < \lvert D_k \rvert < \frac{1}{q_{k+1}}\,.
\end{align}
\end{lemma}

Since $q_k\ge2$ for $k\ge2$, we have from the above lemma that $D_k=(-1)^k\|q_k\alpha\|$ for all $k\ge1$. Also we have that $D_k$ alternates the sign, that is
\begin{equation}\label{s}
  D_kD_{k+1}<0\qquad (k\ge0).
\end{equation}
In particular, in view of \eqref{rel2} and \eqref{s} we have that
\begin{align}
a_{k+1}\lvert D_{k} \rvert + \lvert D_{k+1} \rvert = \lvert D_{k-1} \rvert \qquad (k\ge 1). \label{rel3}
\end{align}
We are now ready to state the full version of the three distance theorem that we will use in this paper.

\begin{theorem}\label{3gap}
Let $\alpha\in\R\setminus\Q$, $[a_0;a_1,a_2,\dots]$ be the continued fraction expansion of $\alpha$, $p_k/q_k$ be the convergents to $\alpha$ and $D_k = q_k\alpha - p_k$.
Then for any $N\in\N$ there exists a unique integer $k\ge0$ such that
\begin{equation}\label{x11}
q_k+q_{k-1}\le N<q_{k+1}+q_k
\end{equation}
and unique integers $r$ and $s$ satisfying
\begin{equation}\label{vb2}
    N=rq_k+q_{k-1}+s,\qquad\enspace
1\le r\le a_{k+1},\qquad\enspace 0\le s\le q_k-1,
\end{equation}
such that the points $\{\alpha\},\{2\alpha\},\ldots,\{N\alpha\}$ split $[0,1]$ into $N+1$ intervals, namely
\begin{equation}\label{int}
J_0:=\big[0,\{n_1\alpha\}\big],\quad J_1:=\big[\{n_1\alpha\},\{n_2\alpha\}\big],~.~.~.~,\quad
J_N:=\big[\{n_N\alpha\},1\big]\,,
\end{equation}
of which
\begin{align}
&\text{$N_A:=N+1-q_k$   ~~~are of length $\delta_A:=|D_k|,$}\label{A}\\
&\text{$N_B:=s+1$ ~~~~~~~~~~are of  length $\delta_B:=|D_{k+1}|+(a_{k+1}-r)|D_k|,$}\label{B}\\
&\text{$N_C:=q_k-s-1$ ~~~~are of length $\delta_C:=\delta_A+\delta_B$.}\label{C}
\end{align}
Furthermore, the unique permutation $(n_1,\dots,n_N)$ of\/ $(1,\dots,N)$ such that
\begin{equation}\label{xx17}
0< \{n_1\alpha\}<\{n_2\alpha\}<\ldots<\{n_{N-1}\alpha\}<\{n_{N}\alpha\}<1
\end{equation}
can be found by setting $n_0:=0$ and then defining
\begin{equation}\label{n_i}
n_{i+1}:=n_i + \Delta_i\qquad\text{for $i=0,\dots,N-1$,}
\end{equation}
where\\[-6ex]
\begin{align}\label{deltacase}
\Delta_{i}:=\left\{\begin{array}{lcl}
                  (-1)^kq_k && \text{if }~n_i\in A\,, \\[0.5ex]
                  (-1)^{k-1}\big(q_{k-1}+rq_{k}\big) && \text{if }~n_i\in B\,, \\[0.5ex]
                  (-1)^{k-1}\big(q_{k-1}+(r-1)q_{k}\big) && \text{if }~n_i\in C
               \end{array}
\right.
\end{align}
and
\begin{align}\label{abcset}
\begin{array}{l}
A:=\big\{n\in\Z\cap[0,N]~:~0\le n+(-1)^kq_k\le N\big\}\,, \\[1ex]
B:=\big\{n\in\Z\cap[0,N]~:~0\le n+(-1)^{k-1}\big(q_{k-1}+rq_{k}\big)\le N\big\}\,,\\[1ex]
C:=\big\{n\in\Z\cap[0,N]~:~n \notin (A \cup B)\big\}
\end{array}
\end{align}
are disjoint subsets of integers; the length of the intervals \eqref{int} is given by
\begin{equation}\label{int2}
|J_i|=\left\{\begin{array}{ll}
                  \delta_A & \text{if }n_i\in A, \\[0ex]
                  \delta_B & \text{if }n_i\in B, \\[0ex]
                  \delta_C & \text{if }n_i\in C.
               \end{array}
\right.
\end{equation}
\end{theorem}

\bigskip

\begin{remark}
Strictly speaking Theorem~\ref{3gap} is not new and can be assembled from published results. The closest versions can be found in \cite{pb}. However, given that Theorem~\ref{3gap} underpins our approach, in \S\ref{3gapproof} we present a complete yet short proof of the result, which also makes this paper self-contained.
\end{remark}

\begin{remark}
Throughout this paper the length of the interval $J_i$ will be referred to as the $i$th gap. It is clear from \eqref{cf1} that $\vert D_{k+1} \rvert < \lvert D_k \rvert$. Hence, $\delta_B < \delta_A$ if and only if $r=a_{k+1}$, and if $r\not=a_{k+1}$, $\delta_A$ will be the smallest (of $\delta_A$, $\delta_B$ and $\delta_C$) gap.
\end{remark}

\begin{remark}
As $\alpha$ is irrational, $\delta_A$, $\delta_B$ and $\delta_C$ are distinct. In fact, while the two gaps $\delta_A$ and $\delta_B$ always appear, $\delta_C$-gaps exist if and only if $s<q_k -1$. There are infinitely many integers $N$ for which there are only two gaps. The structure and the transformation rules for partitioning intervals are studied in details in \cite{tony}. Other in-depth studies involving the two-gaps case can be found in \cite{chev1, chev2, chev3, tony2, sieg}.
\end{remark}

\subsection{Proof of Theorem~\ref{3gap}}\label{3gapproof}

Observe that $(q_k+q_{k-1})_{k\ge0}$ is a strictly increasing sequence of integers starting from $q_0+q_{-1}=1$, whence the existence of $k$ satisfying \eqref{x11} readily follows.
Next, observe that $q_k\le N-q_{k-1}<q_{k+1}+q_k-q_{k-1}=(a_{k+1}+1)q_k$.
Therefore, by division with remainder, $r$ and $s$ satisfying \eqref{vb2} exist and are unique.

Now we verify that $A\cap B=\varnothing$. Indeed, if there existed $n\in A\cap B$, then for even $k$ we would have that
$$
rq_k+q_{k-1} ~\stackrel{n\in B}{\rule{0ex}{2ex}\le} ~n ~\stackrel{n\in A}{\rule{0ex}{2ex}\le} ~ N-q_k~=~rq_k+q_{k-1}+s-q_k \stackrel{s<q_k}{\rule{0ex}{2ex}<} ~rq_k+q_{k-1}\,,
$$
while for odd $k$ we would have that
$$
q_k~\stackrel{n\in A}{\rule{0ex}{2ex}\le} ~ n ~ \stackrel{n\in B}{\rule{0ex}{2ex}\le}  ~N-(rq_k+q_{k-1}) ~ = ~ s~<~q_k.
$$
In both instances we would get a contradiction. In addition, note that $C$ is disjoint from $A$ and from $B$ simply by its definition.

Now we show that
\begin{equation}\label{x17+}
0\le n_{i}\le N\qquad\text{for all } \enspace0\le i\le N\,.
\end{equation}
We use induction on $i$. Clearly, $n_0=0$ satisfies the inequalities. Now assuming that $i<N$ and $0\le n_{i}\le N$ we shall verify that $0\le n_{i+1}\le N$. Indeed, if $n_i\in A\cup B$ this claim follows immediately from the definitions of $A$ and $B$ and \eqref{n_i}. Thus we only need to consider $n_i\in C$. First assume that $k$ is even. Then, since $n_i\not\in A$, we have that $n_i> N-q_k$ and therefore
\begin{align*}
n_{i+1}&\stackrel{\eqref{n_i}}{=}n_i-\big(q_{k-1}+(r-1)q_{k}\big)\ge N-q_k-\big(q_{k-1}+(r-1)q_{k}\big)\\[0.5ex]
&\stackrel{\eqref{vb2}}{=}rq_k+q_{k-1}+s-q_k-\big(q_{k-1}+(r-1)q_{k}\big)=s\ge0
\end{align*}
while clearly $n_{i+1}<n_i\le N$ as required.
Now assume that $k$ is odd. Then, since $n_i\not\in A$, we have that $n_i< q_k$ and therefore
$$
n_{i+1}\stackrel{\eqref{n_i}}{=}n_i+\big(q_{k-1}+(r-1)q_{k}\big)<
q_k+\big(q_{k-1}+(r-1)q_{k}\big)\stackrel{\eqref{vb2}}{=} N-s\le N
$$
while clearly $0\le n_i<n_{i+1}$. This completes the proof of \eqref{x17+}.
The disjointness of $A$, $B$ and $C$ together with \eqref{x17+} imply that $n_i$ given by \eqref{n_i} is well defined.

Next, by
\eqref{A}, \eqref{B}, \eqref{C} and \eqref{abcset}, we have that
\begin{align}\label{habc}
N_A = \# A, \quad N_B = \# B, \quad N_C = \# C
\quad\text{and}\quad
N_A + N_B +N_C = N+1.
\end{align}
Furthermore,
\begin{align}
N_A\delta_A+N_B\delta_B&+N_C\delta_C \ =\ (N-s)\delta_A+q_k\delta_B\nonumber \\[1ex]
&=\ (rq_k+q_{k-1})|D_k|+q_k(|D_{k+1}|+(a_{k+1}-r)|D_k|)\nonumber \\[0.5ex]
&=\ \big(a_{k+1}q_k+q_{k-1}\big)|D_k|+q_k|D_{k+1}|\ \stackrel{\eqref{q_k}}{=}\ q_{k+1}|D_k|+q_k|D_{k+1}| \nonumber \\[0.5ex]
&\stackrel{\eqref{s}}{ = }\ |q_{k+1}D_k-q_kD_{k+1}|
\ \stackrel{\eqref{D_k}}{=}\  |-q_{k+1}p_k+q_kp_{k+1}|\stackrel{\eqref{vb100}}{=}1\label{one}\,.
\end{align}
In particular, it means that $\delta_A$, $\delta_B$ and $\delta_C$ are all strictly less than $1$, except when $N=1$ in which case $N_C=0$ and $\delta_C=1$.

Next we will show that for $0\le i<N$
\begin{equation}\label{x17++}
\{(n_{i+1}-n_i)\alpha\}=\left\{\begin{array}{ll}
                  \delta_A & \text{if }n_i\in A, \\[0ex]
                  \delta_B & \text{if }n_i\in B, \\[0ex]
                  \delta_C & \text{if }n_i\in C.
               \end{array}
\right.
\end{equation}
We shall use \eqref{n_i} and \eqref{deltacase}. In the case $n_i\in A$ we have that
$$
\{(n_{i+1}-n_i)\alpha\}=\{(-1)^kq_k\alpha\}=\{(-1)^k(q_k\alpha-p_k)\}=\{|D_k|\}=\{\delta_A\}=\delta_A\,,
$$
since $0<\delta_A<1$.
In the case $n_i\in B$ we have that
\begin{align*}
\{(n_{i+1}-n_i)\alpha\}&=\{(-1)^{k-1}\big(q_{k-1}+rq_{k}\big)\alpha\}
\stackrel{\eqref{q_k}}{=}\{(-1)^{k-1}\big(q_{k+1}-(a_{k+1}-r)q_{k}\big)\alpha\}\\[0.5ex]
&=\{(-1)^{k-1}\big(q_{k+1}\alpha-p_{k+1}-(a_{k+1}-r)(q_{k}\alpha-p_k)\big)\}\\[0.5ex]
&=
\{|D_{k+1}|+(a_{k+1}-r)|D_{k}|\}=
\{\delta_B\}=\delta_B\,,
\end{align*}
since $0<\delta_B<1$.
Finally, in the case $n_i\in C$ we have that
\begin{align*}
\{(n_{i+1}-n_i)\alpha\}&=\{(-1)^{k-1}\big(q_{k-1}+(r-1)q_{k}\big)\alpha\}\\[0ex]
&\stackrel{\eqref{q_k}}{=}\{(-1)^{k-1}\big(q_{k+1}-(a_{k+1}+1-r)q_{k}\big)\alpha\}\\[0.5ex]
&=\{(-1)^{k-1}\big(q_{k+1}\alpha-p_{k+1}-(a_{k+1}+1-r)(q_{k}\alpha-p_k)\big)\}\\[0.5ex]
&=
\{|D_{k+1}|+(a_{k+1}+1-r)|D_{k}|\}=\{\delta_C\}=\delta_C\,,
\end{align*}
since in this case $N_C\neq0$ and so $0<\delta_C<1$.

\medskip

Now, we prove \eqref{xx17}. First of all note that $0<\{n_1\alpha\}$ since $\alpha$ is irrational. The proof continues by induction. Suppose that $1\le i<N$ and that
$
0<\{n_1\alpha\}<\dots<\{n_i\alpha\}\,.
$
This means that $n_0,n_1,\dots,n_i$ are all different. Therefore, by \eqref{x17++} and the disjointness of $A$, $B$ and $C$, we get that
\begin{equation}\label{eq18}
\sum_{j=0}^{i}\{(n_{j+1}-n_j)\alpha\}< N_A\delta_A+N_B\delta_B+N_C\delta_C\stackrel{\eqref{one}}{=}1\,.
\end{equation}
Therefore,
$$
\textstyle\sum_{j=0}^{i}\{(n_{j+1}-n_j)\alpha\}=\{\sum_{j=0}^{i}(n_{j+1}-n_j)\alpha\}=\{(n_{i+1}-n_0)\alpha\} =\{n_{i+1}\alpha\}\,.
$$
Similarly, $\sum_{j=0}^{i-1}\{(n_{j+1}-n_j)\alpha\}=\{n_i\alpha\}$ and thus
\begin{equation}\label{x+1+7}
\{n_{i+1}\alpha\}=\{(n_{i+1}-n_i)\alpha\}+\{n_i\alpha\}>\{n_i\alpha\}.
\end{equation}
This completes the proof of \eqref{xx17}.
Furthermore, since $n_0=0$, by \eqref{xx17} and\eqref{x17+}, we have that $(n_1,\dots,n_N)$ is the required permutation of $(1,\dots,N)$.

Finally, \eqref{x17++} and \eqref{x+1+7} verify \eqref{int2} for $0\le i\le N-1$, while the facts that  $n_0,\dots,n_N$ are all different and lie in the disjoint sets $A$, $B$ and $C$ put together with \eqref{one} implies \eqref{int2} for $i=N$. In turn, \eqref{A}, \eqref{B} and \eqref{C} are a consequence of \eqref{int2} and \eqref{habc}, namely that $N_A=\#A$, $N_B=\#B$, $N_C=\#C$. The proof of Theorem~\ref{3gap} is thus complete.\\[-2ex]
\hspace*{\fill}$\boxtimes$

\bigskip

\begin{remark}\label{rem7}
Since $0\le n_N\le N$, one can formally define $n_{N+1}$ using \eqref{n_i} and \eqref{deltacase}. This number is easy to compute since amongst $\Delta_0,\dots,\Delta_N$ there are clearly $N_A$ values of $(-1)^kq_k$, $N_B$ values of $(-1)^{k-1}\big(q_{k-1}+rq_{k}\big)$ and $N_C$ values of $(-1)^{k-1}\big(q_{k-1}+(r-1)q_{k}\big)$. Hence, $n_{N+1}$ is exactly
$$
(-1)^kq_kN_A+(-1)^{k-1}\big(q_{k-1}+rq_{k}\big)N_B+(-1)^{k-1}\big(q_{k-1}+(r-1)q_{k}\big)N_C\,.
$$
Substituting the values for $N_A$, $N_B$ and $N_C$ into the above expression and also taking into account that $N=rq_k+q_{k-1}+s$ one readily verifies that
$$
n_{N+1}=0\,.
$$
This means that the sequence $n_i$ can be continued for $i$ beyond $N$ and that it will always be well defined and cyclic. Furthermore it will satisfy the property that
$$
n_{i}=n_j\iff i\equiv j\pmod{N+1}\,.
$$
\end{remark}

\section{Proof of Theorem~\ref{main3}}

\subsection{Semi-homogeneous case}\label{coro2}

We begin by establishing Theorem~\ref{main3} in the case $\{n'\alpha-\gamma\}=0$. Without loss of generality we may assume that $\gamma=n'\alpha$.
In what follows we will use the following basic estimate
\begin{align}\label{log2}
\sum_{t=1}^{T}\frac{1}{t}\le1+\int_{1}^T\frac{dt}{t}=\log T+1\,.
\end{align}
Let $(n_i)_{i\ge 0}$ be the sequence defined by \eqref{n_i} and \eqref{deltacase} extended to all $i$ as described in Remark~\ref{rem7}. Then, there exists a unique integer $\ell\in[0,N]$ such that $n_{\ell}=n'$. By Theorem~\ref{3gap} and Remark~\ref{rem7}, $(n_{\ell},\dots,n_{\ell+N})$ is the permutation of integers $(0,\dots,N)$ such that
\begin{equation}\label{vb55}
0=\{(n_{\ell}-n')\alpha\}<\{(n_{\ell+1}-n')\alpha\}<\dots<\{(n_{\ell+N}-n')\alpha\}<1\,.
\end{equation}
Hence
\begin{equation}\label{sum3}
\sum_{\substack{0\le n\le N\\ n\neq n'}}\frac{1}{\{ (n-n') \alpha \}}=\sum_{j=1}^N\frac{1}{\{ (n_{\ell+j}-n') \alpha \}}\,,
\end{equation}
where the sum on the right of \eqref{sum3} is of decreasing terms. We shall be considering two cases depending on the relative size of $\delta_A$ and $\delta_B$.

\noindent \textbf{Case (i):} $\delta_A < \delta_B$. Note that this means that $r \neq a_{k+1}$ and that $K=k$, where $k$ is as in Theorem~\ref{3gap} and $K$ is as in Theorem~\ref{main3}. Given an index $j$ with $\ell\le j\le \ell+N$, let $h_j$ denote the largest integer such that $0\le h_j\le\max\{0,\ell+N-j\}$ and
\begin{equation}\label{vb101}
  \{(n_{j+i}-n')\alpha\}-\{(n_{j+i-1}-n')\alpha\}=\delta_A\quad\text{for all $i$ with $1\le i\le h_j$.}
\end{equation}
By \eqref{n_i} and \eqref{deltacase}, we have that $\{(n_{j+i}-n')\alpha\}-\{(n_{j+i-1}-n')\alpha\}=\delta_A$ if and only if
$n_{j+i}-n_{j+i-1}=(-1)^kq_k$. Hence $n_{j+h_j}=n_j+h_j(-1)^kq_k$. Now, since $0\le n_{j+h_j},n_{j}\le N$, we must have that
\begin{align}
h_j&=\frac{(-1)^k(n_{j+h_j}-n_j)}{q_k}=\frac{|n_{j+h_j}-n_j|}{q_k}\le \left[\frac{N}{q_k}\right]=:H\,.\label{h_j}
\end{align}
Hence, by \eqref{log2} and the fact that $\delta_A=|D_k|$, we have that the first $H$ terms on the right of \eqref{sum3} in total give us at most
\begin{align}\label{A1}
\sum_{j=1}^{H} \frac{1}{j\delta_A} ~\stackrel{}{<}~ \frac{1}{\lvert D_k \rvert} \left( \log H+1 \right)
~\stackrel{\eqref{cf1}}{<}~ 2q_{k+1}\left( \log \left( \frac{N}{q_k}\right) +1 \right)
\end{align}
which is the right most term in \eqref{e1}, since $k=K$.

Now consider the remaining terms of \eqref{sum3}: $\sum_{j=H+1}^N\{(n_{\ell+j}-n')\alpha\}^{-1}$.
By \eqref{h_j}, amongst any $H+1$ consecutive gaps in \eqref{vb55}, there will be at least one of length $\delta_B$ or $\delta_C$. Therefore, writing $j$ by division with remainder as $t(H+1)+u$ with $0\le u\le H$ we get that
\begin{align*}
\{ (n_{\ell+j}-n') \alpha \} =\sum_{i=1}^{j} \underbrace{ \Big( \{ (n_{\ell+i}-n') \alpha \} - \{ (n_{\ell+i-1}-n') \alpha \} \Big) }_{\mbox{gaps}}\ge t(\delta_B+H\delta_A).
\end{align*}
Observe that
$$
t=\left[\frac{j}{H+1}\right]\le \left[\frac{N}{[N/q_k]+1}\right]\le q_k
$$
and that $t\ge 1$ if and only if $j\ge H+1$.
Then
\begin{align}
\sum_{j=H+1}^N \frac{1}{\{ (n_{\ell+j}-n') \alpha \} } &\le \sum_{t=1}^{q_k} \sum_{u=0}^{H} \frac{1}{t(\delta_B+H\delta_A) }\le \frac{H+1}{\delta_B+H\delta_A}\sum_{t=1}^{q_k}\frac1t\,. \label{sumterms++}
\end{align}
Note that $r=[(N-q_{k-1})/q_k]\le H$. Then, using the explicit values for $\delta_A$ and $\delta_B$ given in Theorem~\ref{3gap} and inequalities \eqref{cf1}, we get that
\begin{equation}\label{l2}
\frac{H+1}{\delta_B+H\delta_A}\le \frac{H+1}{(a_{k+1}+H-r)|D_k|}\le
\frac{H+1}{a_{k+1}|D_k|}\le
\frac{(q_{k+1}+q_k)(H+1)}{a_{k+1}}.
\end{equation}
By \eqref{q_k}, we have that
\begin{equation}\label{l1}
\frac{q_{k+1}+q_k}{a_{k+1}}=q_k+\frac{1}{a_{k+1}}(q_k+q_{k-1})\,.
\end{equation}
Since $r\neq a_{k+1}$ we must have that $a_{k+1}\ge2$ and since $q_{k-1}<q_k$ we then conclude that \eqref{l1} is bounded above by $2q_k$. Putting this together with \eqref{l2} and the facts that $H\le N/q_k$ and $q_k\le N$ we get that
\begin{equation}\label{g1}
\frac{H+1}{\delta_B+H\delta_A}\le 2q_k(N/q_k+1)\le 4N\,.
\end{equation}
This together with \eqref{sumterms++} and \eqref{log2} gives the left most term on the right of \eqref{e1} and thus completes the proof of \eqref{e1}.

\bigskip

\noindent\textbf{Case (ii):} $\delta_A > \delta_B$. Note that this means that $r =a_{k+1}$ and that $K=k+1$.
Since in this case $\delta_B=|D_{k+1}|$ is the smallest gap, we have that
\begin{equation}\label{A2}
\frac{1}{\{ (n_{\ell+1}-n') \alpha \}}\le \frac{1}{\delta_B}\stackrel{\eqref{cf1}}{\le} 2q_{k+2}=2q_{K+1}\le 2q_{K+1}\left( \log \left( \frac{N}{q_{K}}\right) +1 \right)\,,
\end{equation}
which gives the right most term in \eqref{e1}.
Now consider the remaining terms of \eqref{sum3}, namely $\sum_{j=2}^N\{(n_{\ell+j}-n')\alpha\}^{-1}$.

Given an index $j$ with $\ell\le j\le \ell+N$, let $h_j$ denote the largest integer such that $0\le h_j\le\max\{0,\ell+N-j\}$ and
\begin{equation}\label{vb101ii}
  \{(n_{j+i}-n')\alpha\}-\{(n_{j+i-1}-n')\alpha\}=\delta_B\quad\text{for all $i$ with $1\le i\le h_j$.}
\end{equation}
By \eqref{n_i} and \eqref{deltacase}, we have that $\{(n_{j+i}-n')\alpha\}-\{(n_{j+i-1}-n')\alpha\}=\delta_B$ if and only if $n_{j+i}-n_{j+i-1}=(-1)^{k-1}\big(q_{k-1}+rq_{k}\big)$. Hence $n_{j+h_j}=n_j+h_j(-1)^{k-1}\big(q_{k-1}+rq_{k}\big)$. Now, since $0\le n_{j+h_j},n_{j}\le N$, we must have that
\begin{align}
h_j&=\frac{(-1)^k(n_{j+h_j}-n_j)}{q_{k-1}+rq_{k}}=\frac{|n_{j+h_j}-n_j|}{q_{k-1}+rq_{k}}\le \left[\frac{N}{q_{k-1}+rq_{k}}\right]\stackrel{\eqref{vb2}}{=}1\,.\label{h_jii}
\end{align}
Hence, there may not be two consecutive $\delta_B$-gaps in \eqref{vb55} and every other gap in there is of length at least $\delta_A$.
Therefore, writing $j$ by division with remainder as $2t+u$ with $0\le u\le 1$ we get that
\begin{align*}
\{ (n_{\ell+j}-n') \alpha \} =\sum_{i=1}^{j} \underbrace{ \Big( \{ (n_{\ell+i}-n') \alpha \} - \{ (n_{\ell+i-1}-n') \alpha \} \Big) }_{\mbox{gaps}}\ge t(\delta_B+\delta_A).
\end{align*}
Observe that
$$
t=\left[\frac{j}{2}\right]\le \left[\frac{N}{2}\right]\le \left[\frac{q_{k+1}+q_k}{2}\right] < q_{k+1}=q_K
$$
and that $t\ge 1$ if and only if $j\ge 2$.
Then
\begin{align}
\sum_{j=2}^N \frac{1}{\{ (n_{\ell+j}-n') \alpha \} } &\le \sum_{t=1}^{q_K} \sum_{u=0}^{1} \frac{1}{t(\delta_B+\delta_A) }\le \frac{2}{\delta_B+\delta_A}\sum_{t=1}^{q_K}\frac1t\,. \label{sumterms++ii}
\end{align}
Then, using the explicit values for $\delta_A$ and $\delta_B$ given in Theorem~\ref{3gap} and inequalities \eqref{cf1}, we get that
\begin{equation}\label{l2ii}
\frac{2}{\delta_B+\delta_A}\le \frac{2}{|D_k|+|D_{k+1}|}\le
\frac{2}{|D_k|}\stackrel{\eqref{cf1}}{\le} 2(q_{k+1}+q_k).
\end{equation}
Since $r=a_{k+1}$ we must have that $N=rq_k+q_{k-1}+s\ge q_{k+1}$.
Putting this together with \eqref{l2ii} we get that
\begin{equation}\label{g2}
\frac{2}{\delta_B+\delta_A}\le 4N\,.
\end{equation}
This together with \eqref{sumterms++ii} and \eqref{log2} gives the left most term on the right of \eqref{e1} and thus completes the proof of \eqref{e1} in Case (ii).

\bigskip

Finally, note that the condition $n\equiv n'\pmod{q_K}$ excludes the terms contributing to the right most term of \eqref{e1}. Hence \eqref{e2} follows and the proof of Theorem~\ref{main3} is complete in the semi-homogeneous case.

\subsection{The case of arbitrary $\gamma$}

Without loss of generality assume that $0<\gamma<1$ (the case when $\gamma=0$ is covered by the semi-homogeneous case). Then,
\begin{equation}\label{g}
\{n\alpha-\gamma\}=\left\{\begin{array}{ll}
                      \{n\alpha\}-\gamma &\text{if }\{n\alpha\}\ge\gamma\,,\\[1ex]
                      \{n\alpha\}-\gamma+1 &\text{if }\{n\alpha\}<\gamma\,.
                    \end{array}
\right.
\end{equation}

Let $\ell$ be the largest integer in the range $0\le \ell\le N$ such that $\{n_{\ell}\alpha\}<\gamma$. Since $\{0\alpha\}=0<\gamma$, $\ell$ does exist. First assume that $\ell<N$. Set $n'=n_{\ell+1}$. Then, if $\ell+1\le i\le N$ we get that
$$
\{n_i\alpha-\gamma\}\stackrel{\eqref{g}}{=}\{n_i\alpha\}-\gamma \ge \{n_i\alpha\}-\{n_{\ell+1}\alpha\}=\{(n_i-n')\alpha\}\,,
$$
while  if $0\le i\le \ell$ we get that
$$
\{n_i\alpha-\gamma\}\stackrel{\eqref{g}}{=}\{n_i\alpha\}-\gamma+1 \ge
                      \{n_i\alpha\}-\{n_{\ell+1}\alpha\}+1=\{(n_i-n')\alpha\}\,.
$$
Now assume that $\ell=N$. Define $n'=0$. Then for all $0\le i\le N$
$$
\{n_i\alpha-\gamma\}\stackrel{\eqref{g}}{=}\{n_i\alpha\}-\gamma+1 \ge \{n_i\alpha\}=\{(n_i-n')\alpha\}\,.
$$
In either case we have that $T_N(\alpha,\gamma)\le T_N(\alpha,\{n'\alpha\})$ with appropriately defined $n'$ and therefore we finish the proof of the general case of Theorem~\ref{main3} by appealing to the semi-homogeneous case that we have already considered in \S\ref{coro2}.

\section{Generalisations}

The proof of Theorem~\ref{main3} given above can be easily generalised to give a bound on the following more general sum:
$$
T^b_N(\alpha,\gamma):=\sum_{\substack{0\le n\le N\\ n\neq n'}}\frac{1}{\{n\alpha-\gamma\}^b}\,.
$$

\begin{theorem}\label{ab}
\samepage Let $N \in \mathbb{N}$, $\alpha \in \mathbb{R} \backslash \mathbb{Q}$, $p_k/q_k$ denote the convergents to $\alpha$ and $K=K(N,\alpha)$ be the largest integer satisfying $q_K \le N$. Further, let $\gamma \in \mathbb{R}$ and $n'\in[0,N]$ be an integer given by \eqref{n'}. Then
$$
T^b_N(\alpha,\gamma)~\le~ \left\{\begin{array}{ll}
2^{1+b}\zeta(b)N^b + 2^b\zeta(b)q_{K+1}^b&\text{if }b>1\\[3ex]
\frac{2^{1+b}}{1-b}N+\frac{2^b}{1-b}q_{K+1}^b(N/q_K)^{1-b}&\text{if }b<1\,,
                          \end{array}
                          \right.
$$
where $\zeta(b)=\sum_{j=1}^\infty j^{-b}$.
\end{theorem}

\begin{proof}
The arguments are similar to the proof of Theorem~\ref{main3}. We therefore only describe the relevant changes in the calculations made in the proof of Theorem~\ref{main3}. To begin with, note that the general case is again reduced to the semi-homogeneous case. First, assume that $b>1$. Then, in \textbf{Case~(i)}, inequalities \eqref{A1} will be replaced by
\begin{align*}
\sum_{j=1}^H \frac{1}{\{ (n_{\ell+j}-n') \alpha \}^b } \le \sum_{j=1}^{H} \frac{1}{(j\delta_A)^b} ~\stackrel{}{<}~ \frac{1}{|D_k|^b}\sum_{j=1}^Hj^{-b}<\zeta(b)(2q_{K+1})^b=2^b\zeta(b)q_{K+1}^b\,,
\end{align*}
while \eqref{sumterms++} will be replaced by
\begin{align*}
\sum_{j=H+1}^N \frac{1}{\{ (n_{\ell+j}-n') \alpha \}^b } &\le \sum_{t=1}^{q_k} \sum_{u=0}^{H} \frac{1}{(t(\delta_B+H\delta_A))^b }\le \zeta(b)\frac{H+1}{(\delta_B+H\delta_A)^b}\\[1ex]
&= \zeta(b)\left(\frac{H+1}{(\delta_B+H\delta_A)}\right)^b(H+1)^{1-b}\\[2ex]
&\stackrel{\eqref{g1}}{\le} \zeta(b)(4N)^b(H+1)^{1-b}\\[2ex]
&\stackrel{(*)}{\le} \zeta(b)(4N)^b(N/q_K)^{1-b}\\[2ex]
&=4^b\zeta(b)Nq_K^{b-1}~\stackrel{q_K\le N}{\le}~ 4^{b}\zeta(b)N^b\,,
\end{align*}
where $(*)$ is due to the inequalities $b>1$ and $H+1=[N/q_K]+1\ge N/q_K$.

\bigskip

In \textbf{Case~(ii)}, \eqref{A2} will be replaced by
$$
\frac{1}{\{ (n_{\ell+1}-n') \alpha \}^b}\le (2q_{K+1})^b\le 2^b\zeta(b)q_{K+1}^b\,,
$$
since $\zeta(b)>1$, while \eqref{sumterms++ii} will be replaced by
\begin{align*}
\sum_{j=2}^N \frac{1}{\{ (n_{\ell+j}-n') \alpha \} } & \le \sum_{t=1}^{q_K} \sum_{u=0}^{1} \frac{1}{(t(\delta_B+\delta_A))^b }\le \zeta(b)\frac{2}{(\delta_B+\delta_A)^b}\\[1ex]
& =2^{1-b}\zeta(b)\left(\frac{2}{\delta_B+\delta_A}\right)^b~\stackrel{\eqref{g2}}{\le}~
2^{1-b}\zeta(b)\left(4N\right)^b=2^{1+b}\zeta(b)N^b\,.
\end{align*}

Now assume that $b<1$. In this case we will use the following inequality instead of \eqref{log2}
$$
\sum_{t=1}^{T}\frac{1}{t^b}\le1+\int_{1}^T\frac{dt}{t^b}=\frac{T^{1-b}-b}{1-b}\le \frac{T^{1-b}}{1-b}\,.
$$
Then, in \textbf{Case~(i)}, \eqref{A1} will be replaced by
\begin{align*}
\sum_{j=1}^H \frac{1}{\{ (n_{\ell+j}-n') \alpha \}^b } \le \sum_{j=1}^{H} \frac{1}{(j\delta_A)^b} ~\stackrel{}{=}~ \frac{1}{|D_k|^b}\sum_{j=1}^Hj^{-b}<\frac{1}{1-b}(2q_{K+1})^b(N/q_K)^{1-b}\,,
\end{align*}
while \eqref{sumterms++} will be replaced by
\begin{align*}
\sum_{j=H+1}^N \frac{1}{\{ (n_{\ell+j}-n') \alpha \}^b } &\le \sum_{t=1}^{q_k} \sum_{u=0}^{H} \frac{1}{(t(\delta_B+H\delta_A))^b }\\[1ex]
&\le \left(\frac{H+1}{(\delta_B+H\delta_A)}\right)^b(H+1)^{1-b}\frac{q_K^{1-b}}{1-b}\\[1ex]
&\le (4N)^b(2N/q_K)^{1-b}\frac{q_K^{1-b}}{1-b}~=~ \frac{2^{1+b}}{1-b}N\,.
\end{align*}
In \textbf{Case~(ii)}, \eqref{A2} will be replaced by
$$
\frac{1}{\{ (n_{\ell+1}-n') \alpha \}^b}\le (2q_{K+1})^b\le \frac{1}{1-b}(2q_{K+1})^b(N/q_K)^{1-b}\,,
$$
while \eqref{sumterms++ii} will be replaced by
\begin{align*}
\sum_{j=2}^N \frac{1}{\{ (n_{\ell+j}-n') \alpha \} } & \le \sum_{t=1}^{q_K} \sum_{u=0}^{1} \frac{1}{(t(\delta_B+\delta_A))^b }\le 2^{1-b}\left(\frac{2}{\delta_B+\delta_A}\right)^b\frac{q_K^{1-b}}{1-b}\\[1ex]
& \le 2^{1-b}(4N)^b\frac{q_K^{1-b}}{1-b}=\frac{2^{1+b}}{1-b}N^bq_K^{1-b}\le
\frac{2^{1+b}}{1-b}N\,.
\end{align*}
\end{proof}

\begin{remark}
Using the same argument as in Corollary~\ref{main4} one can easily get an analogue of Theorem~\ref{ab} for sums that involve the distance to the nearest integer. Finally, note that using the partial summation formula (see for example \cite{BHV-MEM}) an interested reader can readily deduce upper bounds for the even more general sums of the form
$$
\sum\frac{1}{n^a\{n\alpha-\gamma\}^b}\qquad\text{and}\qquad\sum\frac{1}{n^a\|n\alpha-\gamma\|^b}\,,
$$
where $a\ge0$ and $b>0$.
\end{remark}

\end{document}